\documentclass[onefignum,onetabnum]{siamart220329}

\usepackage[OT1]{fontenc}
\usepackage{amsmath}
\usepackage{amssymb}

\usepackage{mathtools}
\usepackage{mathpazo}
\usepackage[mathpazo]{flexisym}
\usepackage{breqn}
\usepackage{nomencl}
\usepackage{tablefootnote}
\usepackage{ulem}
\usepackage{hyperref}
\usepackage{cleveref}
\usepackage[most]{tcolorbox}
\usepackage{tikz-cd}

\usepackage{lipsum}
\usepackage{amsfonts}
\usepackage{graphicx}
\usepackage{epstopdf}
\usepackage{algorithmic}
\ifpdf
  \DeclareGraphicsExtensions{.eps,.pdf,.png,.jpg}
\else
  \DeclareGraphicsExtensions{.eps}
\fi

\newsiamremark{remark}{Remark}
\newsiamremark{hypothesis}{Hypothesis}
\crefname{hypothesis}{Hypothesis}{Hypotheses}
\newsiamthm{claim}{Claim}
\newsiamremark{myprob}{Problem}
\headers{High order biorthogonal}{T. Haubold, S. Beuchler, and J. Sch\"oberl}

\title{High order biorthogonal functions in \texorpdfstring{$H(\Curl)$}{}\thanks{Submitted to the editors \today.}\funding{The work of the second author is funded by the Deutsche Forschungsgemeinschaft (DFG) under Germany’s Excellence Strategy within the Cluster of Excellence PhoenixD (EXC 2122, Project ID 390833453).}}

\author{Tim Haubold\thanks{Leibniz University Hannover, Hannover, Germany 
  (\email{haubold@ifam.uni-hannover.de}).}
\and Sven Beuchler\footnotemark[2]
\and Joachim Sch\"oberl\thanks{Institute for Analysis and Scientific Computing, TU Wien, Vienna, Austria}}

\usepackage{amsopn}

\usepackage[most]{tcolorbox}
\usepackage{tikz-cd}
\tikzset{%
  symbol/.style={
    draw=none,
    every to/.append style={
      edge node={node [sloped, allow upside down, auto=false]{$#1$}}
    },
  },
}
\tcbuselibrary{theorems}
\newtcolorbox[auto counter,number within=section,crefname={table}{tables}]{funbox}[2][]{%
nofloat,
colframe=purple!75!black,
colback=purple!5!white,
coltitle=yellow,
subtitle style={boxrule=0.6pt,colback=purple!15!white,fontupper=\color{black}},
fonttitle=\bfseries,
fonttitle=\bfseries,
title=Table~\thetcbcounter: #2,label={#1}}

\newtcolorbox[auto counter,number within=section,crefname={table}{tables}]{recbox}[2][]{%
colframe=violet!75!black,colback=violet!5!white,coltitle=yellow,subtitle style={boxrule=0.4pt,
colback=violet!75!tcbcolbacklower} ,
fonttitle=\bfseries,fonttitle=\bfseries,
title=Table~\thetcbcounter: #2,label={#1}}
\tcbuselibrary{skins}
\usepackage{algorithm}
\usepackage{subcaption}
\usepackage{aligned-overset}
\usepackage{algorithmic}
\usepackage{epsfig}
\usepackage{makecell}
\pgfmathsetmacro{\a}{-2}%
\pgfmathsetmacro{\b}{2}
\pgfmathsetmacro{\h}{3}%

\usetikzlibrary{calc}%
\DeclarePairedDelimiter\abs{\lvert}{\rvert}%
\DeclarePairedDelimiter\norm{\lVert}{\rVert}%
\makeatletter
\let\oldabs\abs
\def\abs{\@ifstar{\oldabs}{\oldabs*}}
\let\oldnorm\norm
\def\norm{\@ifstar{\oldnorm}{\oldnorm*}}
\makeatother
\newcommand{\mathsym}[1]{{}}
\newcommand{\unicode}[1]{{}}

\newcommand{\lhat}{\widehat{L}}
\newcommand{\phat}{\widehat{P}}
\newcommand{\dufx}{\left(\frac{2x}{1-y}\right)}

\newcommand{\Grad}{\nabla}
\DeclareMathOperator{\Curl}{curl}
\DeclareMathOperator{\Div}{div}
\DeclareMathOperator{\Id}{id}
\newcommand{\dx}{\, \mathrm{d}x}
\newcommand{\dy}{\, \mathrm{d}y}
\newcommand{\dz}{\, \mathrm{d}z}

\newcommand{\deta}{\, \mathrm{d}\eta}

\newcommand{\R}{\mathbb{R}}

\newcommand{\V}{\mathbb{V}}
\newcommand{\ned}{\mathcal{N}}

\newcommand{\sm}{\scalebox{0.65}[1.0]{\( - \)}}
\newcommand{\spl}{\scalebox{0.65}[1.0]{\( + \)}}

\ifpdf
\hypersetup{
  pdftitle={High order biorthogonal functions},
  pdfauthor={T. Haubold and S. Beuchler and Joachim Sch\"oberl}
}
\fi

\begin{document}

\maketitle

\begin{abstract}
From the literature, it is known that the choice of basis functions in $hp$-FEM  heavily influences the computational cost in order to obtain an approximate solution. 
Depending on the choice of the reference element, suitable tensor product like basis functions of
Jacobi polynomials with different weights lead to optimal properties due to condition number and sparsity.
This paper presents biorthogonal basis functions to the primal basis functions mentioned above.
The authors investigate hypercubes and simplices as reference elements, as well as the cases of $H^1$ and
$H(\Curl)$. The functions can be expressed sums of tensor products of Jacobi polynomials with maximal
two summands.
 \end{abstract}



\section{Introduction}
It is well known, that $hp$ finite element methods (fem) often exhibit exponential convergence rates depending on the domain and the boundary, see e.g. \cite{Szabo,Schwab,melenk}, if the exact solution of the
underlying partial differential equation is (locally) sufficiently smooth.
One of the most important algorithmic parts of a $hp$-fem method is the choice of basic functions. 
It has been shown in \cite{Babuka1989,Maitre} that the choice directly influences the condition number of the local and global matrices. 
Nodal basis functions, e.g.  based on Lagrangian polynomials, exhibit an exponential growth in the condition number.  We also mention the related spectral element method \cite{Karniadakis}. 
For efficient assembly routines, see \cite{Eibner}.\\
An alternative to nodal basis functions are modal basis function, e.g. based on Bernstein or classical orthogonal polynomials.
The former choice has the advantage that assembly routines in optimal complexity exist and, by their natural connection to \texttt{NURBS}, can easily be visualized, see \cite{Ainsworth1,Ainsworth2}. 
On the other hand, the latter choice yields in overall better condition numbers. Furthermore, in case of a polygonal domain and piecewise constant material functions, they yield in sparse element matrices and can be assembled in optimal complexity as well, see \cite{Haubold}.\\
In this paper, we use basis functions based on integrated orthogonal polynomials. This choice goes back to Szab\'{o} and Babu\v{s}ka, see e.g. \cite{Szabo}. For simplices, an analogue choice was introduced by \cite{Dubiner,Karniadakis95,beuchler}. Due to different variational formulations, one naturally derives the different function spaces $H^1,H(\Curl)$ and $H(\Div).$ These spaces are connected by an exact sequence, the so called \textit{De-Rham-complex}. For the weak gradient, curl and divergence operator, these complexes read in $2D$ as either
\begin{equation*}
 \R \overset{\Id}{\to} H^1(\Omega) \overset{\Grad}{\to} H(\Curl,\Omega) \overset{\Curl}{\to} L^2(\Omega) ,\end{equation*}
 or
 \begin{equation*}
 \R \overset{\Id}{\to} H^1(\Omega)  \overset{\Curl}{\to} H(\Div,\Omega) \overset{\Div}{\to} L^2(\Omega),\end{equation*}
 and in $3D$ as
 \begin{equation*}
 \R \overset{\Id}{\to} H^1(\Omega) \overset{\Grad}{\to} H(\Curl,\Omega) \overset{\Curl}{\to} H(\Div,\Omega) \overset{\Div}{\to} L^2(\Omega),
\end{equation*}
where $\Omega$ is an arbitrary bounded, simply connected domain. 
These complexes and their affiliated discrete spaces are discussed e.g. in \cite{monk,demkowiczbookI,demkowiczbookII}. A set of high order basis functions on simplices based on barycentric coordinates was introduced in \cite{Gopalakrishnan2005}.
The high-order basis function, which we will consider, are based on Jacobi polynomials and were systematically introduced e.g. in \cite{zaglmayrdiss}.
This construction principle was also used in \cite{AINSWORTH20016709,beuchlerCurl,beuchlerdiv} 
to modify the weights of the chosen Jacobi polynomials due to optimization of sparsity pattern and condition number. A connected construction was given by Fuentes and coworkers \cite{Fuentes}.
This ansatz avoids the problem of element orientation and thus can combine different element types, like simplices, hexahedron, and even pyramids.

The aim of this paper is to give new high order dual functions for the $H^1$ and $H(\Curl)$ basis functions defined in \cite{beuchlerpillwein} and \cite{beuchlerCurl}.
Consider an arbitrary basis $\lbrace\phi_i\rbrace_i$, then the respective ($L^2$) dual functions $\lbrace\psi_j\rbrace_j \in L^2(\Omega)$ are given by the relation
\begin{equation*}
    \langle \phi_i,\psi_j \rangle_{L^2} = \delta_{ij},
\end{equation*}
where $\delta_{km}=1$ for $k=m$ and $\delta_{km}=0$ else. Dual functions are used in defining interpolation operators, see e.g. \cite[Chap. 7]{mallet}, or transfer operators between finite elements spaces, see e.g. \cite{Wohlmuth98,Wohlmuth2001b}.
Another application is the determination of starting values for time dependent parabolic problems.

Consider for example some function $t(\vec{x}),$ which we want to approximate by our basis functions $\phi_i,$ i.e.
\begin{equation*}
    t(\vec{x}) = \sum_{i=1}^N \alpha_i \phi_i, 
\end{equation*}
where $N = \dim(\phi).$ 
This best approximation problem is solved by multiplication with test functions $v \in L^2(\Omega)$ and integration thereof. The following linear system needs to be solved:
\begin{equation*}
    \begin{aligned}
        \int_{\Omega} t(\vec{x}) v(\vec{x}) \mathrm{d}\vec{x} = \sum_{i=1}^N\int_{\Omega} \alpha_i \phi_i(\vec{x}) v(\vec{x}) \mathrm{d}\vec{x} \quad \forall v \in L^2(\Omega).
    \end{aligned}
\end{equation*}
The choice $v = \phi_i$ for all $i$ would lead into a dense or almost dense system. On the other hand, the choice of biorthogonal functions leads to a diagonal system. 
There are also purely biorthogonal polynomial systems, see e.g. \cite{Dunkl}.\\
An algorithmic implementation for the $H^1$ dual functions can be found in the finite element software \texttt{Ngsolve} \cite{ngsolve}, which uses a projected based interpolation, see also \cite{demkowiczbookI,demkowiczbookII}.\\
In $2D$ the general problem reads:
\begin{myprob}{}{}
Find $u_{hp} \in \V_{hp}$ such that
\begin{align}
    u_{hp}(\lambda) &= u(\lambda) \quad &&\forall \text{ vertices } \lambda,\\
    \int_E u_{hp} v&= \int_E uv &&\forall v \in \mathcal{P}^{p-2} \text{ or } \mathcal{Q}^{p-2} \quad \forall \text{ edges } E,\\
    \int_{Q/T} u_{hp} v &= \int_{Q/T}uv &&\forall v \in \mathcal{P}^{p-3} \text{ or } \mathcal{Q}^{p-3}\quad\forall \text{ triangles/quadrilaterals } T \label{proInt}.
\end{align}
\end{myprob}
A similar approach is valid in $3D.$ 

Depending on the choice of the primal spaces $\V_{hp}$, the authors present dual basis functions for
$H^1$ and $H(\Curl)$ on several elements. Some cases, in particular on hypercubes in all dimensions
as well as in $H^1$ are quite simple and should be known in the $hp$-fem community.
Nevertheless, we decided to present also these cases since they give the reader an impression
of the construction principle for the difficult ones in $H(\Curl)$ on simplices.
This major novelty of this contribution is the development of closed expressions 
of dual functions in terms of Jacobi polynomials. 
In our publications, the primal basis functions are tensor products of Legendre and integrated
Legendre polynomials for the hypercubes elements. On simplices, we used the functions \cite{beuchlerpillwein} and \cite{beuchlerCurl}.

The outline of this paper is as follows. 
Section \ref{sec:Prem} introduces and summarizes the required properties of Jacobi polynomials.
In section \ref{sec:H1} the dual basis functions for $H^1$ are stated.
The main part of this paper is devoted to section \ref{sec:Hcurl}, where the biorthogonal functions
for $H(\Curl)$ are derived. Section \ref{sec:Con} summarizes the main results of this paper.

\section{\label{sec:Prem}Preliminaries}
We start with the definition of the required orthogonal polynomials.
For $n\in\mathbb{N}$, $\alpha,\beta>-1$, let 
\begin{equation}\label{def:Jacobi}
    P_{n}^{(\alpha,\beta)}(x)=\frac{1}{2^n n!(1-x)^\alpha(1+x)^\beta} 
    \frac{\mathrm{d}^n}{\mathrm{d}x^n} (1-x)^\alpha(1+x)^\beta (x^2-1)^n
\end{equation}
be the $n$-Jacobi polynomial with respect to the weight $\omega(x)=(1-x)^\alpha(1+x)^\beta$.
Moreover, the integrated Jacobi polynomials are given by
\begin{equation}
    \label{def:IntJac}
    \hat{P}_{n}^{\alpha}(x)=\int_{-1}^x  P_{n-1}^{(\alpha,0)} (t) \;\mathrm{d}t, \quad \hat{P}_1^{\alpha}(x) = 1
\end{equation}
for $n\geq 0$ and $\beta=0$.
In the special case $\alpha=\beta=0$, one obtains 
\begin{equation}
    \label{eq:defLeg}
    L_n(x)=P^{(0,0)}_n(x) \quad\textrm{and}\quad \hat{L}_n(x)=\hat{P}_n^0(x)
\end{equation}
the Legendre and integrated Legendre polynomials, respectively.
The Jacobi polynomials form an orthogonal system in the weighted $L_{2,\omega}$ scalar product
\begin{equation}
\label{Orthogonality2}
I_{n,m}^{(\alpha,\beta)}=\int_{-1}^1 \omega(x) P_{n}^{(\alpha,\beta)}(x)P_{m}^{(\alpha,\beta)}(x) \;\mathrm{d}x
= \delta_{nm} \frac{{2}^{\alpha+\beta+1}}{(2n+\alpha+\beta+1} 
\frac{\Gamma(n+\alpha+1)\Gamma(n+\beta+1)}{n! \Gamma(n+\alpha+\beta+1)},
\end{equation}
where $\Gamma(\cdot)$ denotes the Gamma function,
see e.g. \cite{Askey}.
In this publication, the cases $\beta\in\{0,1\}$ are of special interest.
The relation \eqref{Orthogonality2} simplifies to
\begin{equation}
    \label{Orthogonality}
    I_{n,m}^{(\alpha,0)}=\delta_{nm} \frac{2^{\alpha+1}}{(2n+\alpha+1)} \quad\textrm{and}\quad
    I_{n,m}^{(\alpha,1)}=\delta_{nm} \frac{2^{\alpha+2}}{(2n+\alpha+2)} \frac{(n+1)}{(n+\alpha+1)},
\end{equation}
respectively.
Note that the integrated Legendre and Jacobi polynomials can be written as Jacobi polynomials by
using the relations, 
\begin{equation}\label{dual:IntLeg}
    \lhat_i(x) = \frac{(x^2-1)}{2(i-1)} P_{i-2}^{(1,1)}(x)
    \quad\textrm{and}\quad  \phat^{\alpha}_i(x) = \frac{(1+x)}{n} P^{(\alpha - 1,1)}_{i-1}(x),
\end{equation}
see e.g. \cite{Szego}.
\section{\label{sec:H1}Dual function in \texorpdfstring{$H^1$}{}}
We will start by introducing the $H^1$ dual functions for the quadrilateral and for the triangular case.
\subsection{\texorpdfstring{$H^1$}{} dual functions on the quadrilateral}
Consider the master element $\square = (-1,1)^2.$ 
The standard interior basis functions are $u^\square_{ij}(x,y) = \lhat_i(x) \lhat_j(y),$ see \cite{Szabo}. 
To find the dual functions, we use \eqref{dual:IntLeg}.
The dual function $\hat{b}^\square_{kl}(x,y)= b^\square_k(x) b^\square_l(y)$ has to satisfy the relation
\begin{equation}
    \int_\square u^\square_{ij}(x,y) \hat{b}^\square_{kl}(x,y) \dx = \int_{-1}^{1} \lhat_i(x) b^\square_k(x) \dx \int_{-1}^1 \lhat_j(y) b^\square_l(y) \dy = c_{ijkl} \delta_{i,k} \delta_{j,l}.
\end{equation}
By inserting the Jacobi polynomials, we directly notice the orthogonality relation \eqref{Orthogonality} 
for $\alpha=\beta=1$. Thus, one obtains
\begin{equation*}
    \tilde{c} \int_{-1}^{1} (x^2-1) P_{i-2}^{(1,1)}(x) \hat{b}^\square_k(x) \dx \int_{-1}^1 (y^2-1) P_{i-2}^{(1,1)}(y) \hat{b}^\square_l(y) \dy = c_{kl} \delta_{i,k} \delta_{j,l}.
\end{equation*}
This motivates the choice
\begin{equation} \label{eq:dualquad}
    b^\square_k(x) = \chi_k P^{(1,1)}_{k-2}(x) \text{ and } b^\square_l(y) = \chi_l P^{(1,1)}_{l-2}(y).
\end{equation}
with some factors $\chi_l$.
The extension to the three-dimensional case is straightforward.
\subsection{\texorpdfstring{$H^1$ dual function on the simplex}{}}
We now apply the same strategy to the simplicial case. The triangular basis functions, given in \cite{beuchler}, are used
on the triangle $\triangle$ with vertices
$(\sm1,\sm1),(1,\sm1),(0,1)$. For the tetrahedron $\blacktriangle$ with vertices $(\sm1,\sm1,\sm1)$, $(1,\sm1,\sm1)$, $(0,1,\sm1)$ and $(0,0,1),$ the basis functions of \cite{beuchler3} are used.
The interior bubbles are given as
\begin{equation}\label{eq:TrigFace}
\begin{aligned}
u^\triangle_{ij}(x,y)& = \lhat_i\left(\frac{2x}{1-y}\right) \left(\frac{1-y}{2}\right)^i \phat_{j}^{2i}(y)\\
u^\blacktriangle_{ijk}(x,y,z)& = \lhat_i\left(\frac{4x}{1-2y-z}\right) \left(\frac{1-2y-z}{4}\right)^i \phat_{j}^{2i}\left( \frac{2y}{1-z}\right) \left(\frac{1-z}{2}\right)^j \phat_{k}^{2i+2j}(z),
\end{aligned}
\end{equation}
for the triangle and tetrahedron, respectively.

Using \eqref{dual:IntLeg}, $u^\triangle_{ij}(x,y)$ is rewritten as
\[u^\triangle_{ij}(x,y) = c \frac12 \left(\dufx^2-1\right) P_{i-2}^{(1,1)}\left(\frac{2x}{1-y}\right) \left(\frac{1-y}{2}\right)^i \left(\frac{1+y}{2}\right) P_{j-1}^{(2i-1,1)}(y),\]
with some known constant $c.$
As before we search for $\hat{b}^\triangle_{kl}(x,y) = b^\triangle_k\left(\frac{2x}{1-y}\right)b_{kl}^\triangle(y)= b^\triangle_k\left(\eta\right)b_{kl}^\triangle(y),$
where $\eta = \frac{2x}{1-y}.$
Using the Duffy transformation we write down the biorthogonality condition as
\begin{equation}\label{Dual:cond}
\begin{aligned}
    \int_\triangle &u^\triangle_{ij}(x,y) \hat{b}_{kl}^\triangle(x,y) \dx\\ &= c \int_{-1}^1 \left(\frac{\eta^2-1}{2}\right) P_{i-2}^{(1,1)}(\eta) b_k^\triangle(\eta) \mathrm{d}\eta \int_{-1}^{1} \left(\frac{1-y}{2}\right)^{i+1} P^{(2i-1,1)}_{j-1}(y) b^\triangle_{kl} (y) \dy.  
\end{aligned}
\end{equation}
Again this motivates the choice
\begin{equation*}
    b^\triangle_k(\eta) = P_{k-2}^{(1,1)}(\eta)\quad \text{and } b_{kl}^{\triangle}(y) = \left(\frac{1-y}{2}\right)^{k-2} P^{(2k-1,1)}_{l-1}(y).
\end{equation*}
Normalizing the dual functions means that the system matrix is again the identity matrix. We summarize in the following lemma:
\begin{lemma}[\texorpdfstring{$H^1$}{} dual functions on a triangle] \label{Dual:LemTrig}
The interior functions $u^\triangle_{ij}(x,y)$ as in \eqref{eq:TrigFace} and 
\begin{equation}\label{Dual:trig}
    \hat{b}^\triangle_{kl}(x,y) = P_{k-2}^{(1,1)}\left(\frac{2x}{1-y}\right) \left(\frac{1-y}{2}\right)^{k-2} P^{(2k-1,1)}_{l-1}(y) \quad \forall \, k\geq 2, i\geq 1
\end{equation}
are a biorthogonal system on $\triangle,$ i.e. 
\begin{equation*}
    \int_\triangle u_{ij}^\triangle(x,y) \hat{b}^\triangle_{kl}(x,y) \dx \dy = c \delta_{ik} \delta_{jl}
\end{equation*}
On the tetrahedron $\blacktriangle$ the interior functions $u^\blacktriangle_{ijk}$ and 
\begin{equation*}
\begin{aligned}
    b^\blacktriangle_{lnm}&(x,y,z)\\
    &= P^{(1,1)}_{l-2} \left(\frac{4x}{1-2y-z}\right) \left(\frac{1-2y-z}{4}\right)^{(l-2)}P_{n}^{(2i-1,1)}\left(\frac{2z}{1-y}\right)^{(n-1)} P_{m}^{(2i+2j-1,1)}(z)
\end{aligned}
\end{equation*}
are biorthogonal, i.e. 
\begin{equation}\label{Dual:tet}
    \int_{\blacktriangle} u_{ijk}^\blacktriangle(x,y,z) b_{lnm}^\blacktriangle(x,y,z) \dx \dy \dz = \tilde{c} \delta_{il} \delta_{jn} \delta_{km}.
\end{equation}
\end{lemma}
\begin{proof}
The biorthogonality follows by inserting \eqref{Dual:trig} in \eqref{Dual:cond}. Analogously, the biorthogonality for the tetrahedron follows by \eqref{Dual:tet}.
\end{proof}
\section{\label{sec:Hcurl}Dual functions in \texorpdfstring{$H(\Curl)$}{}}
In a next step, dual functions for $H(\Curl)$ will be derived. We apply the following notation a function $u$ denotes a $H^1$ basis function, $v^{Q,a}$ a  $H(\Curl)$ basis function of type $a = I,II,III$ on a reference element $Q.$ Here $(i,j)$ are the indices of the basis functions, while $(k,l)$ are the indices of the dual functions in $2D,$ or $(i,j,k),$ and $(l,m,n)$ respectively in $3D.$ Furthermore $p$ denotes either the total or the maximal polynomial degree, depending on the reference element.

Finding dual functions for $H(\Curl)$ functions is more complicated. Not only are the shape functions vectorial, but they also appear in multiple types. Our goal is to find all dual functions which are orthogonal to the corresponding type of $H(\Curl)$ and additionally are zero for all other types.
\subsection{Quadrilateral basis}
Recall that the $H(\Curl)$ face shape functions on the quadrilateral $\square=(-1,1)^2$ are
\begin{equation}\label{Dual:QuadCurlBasis}
\begin{aligned}
    v_{ij}^{\square,I}(x,y) &= \Grad \left(\lhat_i(x) \lhat_j(y)\right) = \begin{pmatrix} L_{i-1}(x) \lhat_{j}(y) \\ \lhat_{i}(x) L_{j-1}(y)\end{pmatrix},\\
    v_{ij}^{\square,II}(x,y) &= \begin{pmatrix} L_{i-1}(x) \lhat_{j}(y) \\ -\lhat_{i}(x) L_{j-1}(y)\end{pmatrix},\\
    v_{i}^{\square,III}(x,y) &= \begin{pmatrix} \lhat_{i}(y) \\ 0\end{pmatrix},\quad 
    v_{i+p}^{\square,III}(x,y) = \begin{pmatrix} 0 \\ -\lhat_{i}(x)\end{pmatrix},
\end{aligned}
\end{equation}
for $2 \leq i,j \leq p,$ see \cite{AINSWORTH20016709,Zaglmayr}. 
After linear combination we see, that we can also define the auxiliary interior functions as 
\begin{align}
    \label{eq:HcurlI}\tilde{v}_{ij}^{\square,I}(x,y) &= \Grad(\lhat_i(x)) \lhat_j(y) = \begin{pmatrix} L_{i-1}(x) \lhat_{j}(y) \\0 \end{pmatrix},\quad \text{ for } 1 \leq i \leq p, 2 \leq j \leq p,\\
    \label{eq:HcurlII}\tilde{v}_{ij}^{\square,II}(x,y) &= \lhat_i(x)\Grad(\lhat_j(y))=\begin{pmatrix} 0 \\ \lhat_{i}(x) L_{j-1}(y)\end{pmatrix},\quad \text{ for } 1 \leq j \leq p, 2 \leq i \leq p.
\end{align}
Note that the functions of type $III$ are now a special case of the new type $I$ and $II$ for $i=1$ and $j=1,$ respectively.
By application of \eqref{dual:IntLeg} we find the biorthogonal functions to the auxiliary functions.
\begin{definition}[Dual functions on a quadrilateral]
Let
\begin{equation}\label{eq:BiQuadCurl}
    \begin{aligned}
    \tilde{b}_{kl}^{\square,I}(x,y) &\coloneqq \begin{pmatrix} L_{k-1}(x) P^{(1,1)}_{l-2}(y) \\0 \end{pmatrix},\quad \text{ for } 1 \leq k \leq p, 2 \leq l \leq p,\\
    \tilde{b}_{kl}^{\square,II}(x,y) &\coloneqq \begin{pmatrix} 0 \\ P^{(1,1)}_{k-2}(x) L_{l-1}(y)\end{pmatrix},\quad \text{ for } 1 \leq l \leq p, 2 \leq k \leq p.
\end{aligned}
\end{equation}
\end{definition}
\begin{corollary}\label{Dual:QuadCurl}
For $\tilde{v}^{\square,I}$ and $\tilde{v}^{\square,II}$ as in \eqref{eq:HcurlI} and \eqref{eq:HcurlII} and the functions $\tilde{b}^{\square,I}$ and $\tilde{b}^{\square,II}$ as in \eqref{eq:BiQuadCurl} are biorthogonal, i.e.
\begin{equation*}
    \int_\square \tilde{v}^{\square,\omega_1}_{ij}(x,y) \tilde{b}^{\square,\omega_2}_{kl}(x,y) = c \delta_{ik} \delta_{jl} \delta_{\omega_1,\omega_2}
\end{equation*}

\end{corollary}
\begin{proof}The proof follows as in \cref{Dual:LemTrig}.\end{proof}
We want to apply the dual functions from \cref{Dual:QuadCurl} to \eqref{Dual:QuadCurlBasis}. For this, we state the following lemma:
\begin{lemma}\label{modBi}
Let the functions $\phi_i^{t}$ and $\psi_j^{r}$ satisfy
\[
\langle \phi_i^{t}, \psi_j^{r} \rangle=d_{i,t}  \delta_{ij} \delta_{tr}
\]
with a constant $d_{i,t}>0$. Moreover, let $\tilde{\phi}_{i}^{t}=\sum_{s=1}^k a_{t,s} \phi_i^s$ and $\tilde{\psi}_{j}^{r}=\sum_{s=1}^k b_{r,s} \psi_j^s$. 
If $ADB^\top=I$, where
$A=[ a_{t,s}]_{t,s=1}^k$, 
$B=[b_{r,s}]_{r,s=1}^k,$ and $D=\mathrm{diag}[d_{i,s}]_{s=1}^k,$ then
\[
\langle \tilde{\phi}_i^{t}, \tilde{\psi}_j^{r} \rangle=\delta_{ij} \delta_{rt}.
\]
\end{lemma}
\begin{proof}
The proof is elementary.
\end{proof}
We apply now lemma \ref{modBi} to find the biorthogonal linear combination of \eqref{eq:BiQuadCurl}. For \eqref{Dual:QuadCurlBasis} it is $A=\begin{bmatrix} 1 & 1 \\ 1 & -1\end{bmatrix}$ and $D$ is given by the relations
\begin{equation*}
\begin{aligned}
   \langle v^{\square,II}_{ij}, \tilde{b}^{\square,I}_{kl} \rangle =  \langle v^{\square,I}_{ij}, \tilde{b}^{\square,I}_{kl} \rangle &= c_I \, \delta_{ik} \delta_{jl} 
   \textrm{ and }
   -\langle v^{\square,II}_{ij}, \tilde{b}^{\square,II}_{kl} \rangle =  \langle v^{\square,I}_{ij}, \tilde{b}^{\square,II}_{kl} \rangle&= c_{II} \, \delta_{ik} \delta_{jl}
\end{aligned}
\end{equation*}
where $c_I\neq0$ and $c_{II}\neq0$ are constants depending on $i,j,k,l.$ 
This motivates the choices
\begin{equation*}
    b^{\square,II}_{ij}(x,y) =  \frac{\tilde{b}_{ij}^{\square,II}(x,y)}{c_{II}} -   \frac{\tilde{b}_{\square,ij}^I(x,y)}{c_I},\quad
    b^{\square,I}_{ij}(x,y) = \frac{\tilde{b}_{ij}^{\square,I}(x,y)}{c_{II}}+ 
    \frac{\tilde{b}_{ij}^{\square,II}(x,y)}{c_I}.
\end{equation*}
Lemma \ref{modBi} guarantees the orthogonality relations
\[
\langle v^{\square,I}_{ij}, b^{II}_{kl} \rangle=0 =\langle v^{\square,I}_{ij}, b^{II}_{kl} \rangle
\qquad \forall i,j,k,l \geq 2.
\]
Orthogonality of $v_i^{III}$ and $v_{i+p}^{III}$ to $b^{I}_{kl}$ and $b^{II}_{kl}$ is trivial.
We summarize in the following lemma.
\begin{lemma}
Let $v^{\square,I}_{ij}, v^{\square,II}_{ij}, v^{\square,III}_{i}$ and $v_{i+p}^{\square,III}$ be as in \eqref{Dual:QuadCurlBasis}, and $\tilde{b}^{\square,I}_{kl}$ and $\tilde{b}^{\square,II}_{kl}$ as in \eqref{eq:BiQuadCurl}. Furthermore, let \begin{equation}\label{Dual:CoefQuad}
    \alpha_{ij} = \frac{1}{8} i(2i-1)(2j-1).  
\end{equation}
Then the functions
\begin{align*}
    b^{\square,I}_{ij}(x,y) &= \alpha_{ij} \, \tilde{b}_{ij}^{\square,I}(x,y) - \alpha_{ji} \, \hat{b}_{ij}^{\square,II}(x,y),\\
    b^{\square,II}_{ij}(x,y) &=   \alpha_{ij} \, \tilde{b}_{ij}^{\square,I}(x,y) + \alpha_{ji} \,\tilde{b}_{ij}^{\square,II}(x,y),\\
    b^{\square,III}_{i}(x,y) &= \tilde{b}^{\square,I}_{1l}(x,y), \quad b^{\square,III}_{i+p}(x,y) = \tilde{b}^{\square,II}_{1l}(x,y),
\end{align*}
are biorthogonal to \eqref{Dual:QuadCurlBasis}.
\end{lemma}
\begin{proof}
Biorthogonality was already shown above. For the coefficients in \eqref{Dual:CoefQuad} apply the usual orthogonality results \eqref{Orthogonality} for $\alpha=\beta=0$ and $\alpha=\beta=1$.
One obtains
\[
c = \langle v^I_{ij}, b^I_{ij} \rangle = \int_{-1}^1 L_{i-1}(x) L_{i-1}(x) \dx\int_{-1}^{1} \frac{y^2-1}{2(j-1)} P_{j-2}^{(1,1)}(y) P_{j-2}^{(1,1)}(y) \dy =\alpha^{-1}_{ji}. 
\]
The coefficient $\alpha_{ij}$ is computed analogously.
\end{proof}
\subsection{Triangular case}
The triangular case is more complicated. 
Following \cite{beuchlerCurl}, the basis functions of $H(\Curl)$ on the reference triangle $\triangle$ are given as
\begin{equation}\label{Dual:BasisTrig}
\begin{aligned}
    v^{\triangle, I}_{ij}(x,y) &= \Grad(u^{\triangle}_{ij}(x,y))= \Grad(f_i(x,y)) g_{ij}(y) + f_i(x,y) \Grad(g_{ij}(y))\\
    v^{\triangle, II}_{ij}(x,y) &= \Grad(f_i(x,y)) g_{ij}(y) - f_i(x,y) \Grad(g_{ij}(y))\\
    v^{\triangle,III}_{1j}(x,y) &= \Grad(f_1(x,y)) \phat^{3}_j(y),
\end{aligned}
\end{equation}
where $f_i(x,y) = \lhat_i(\frac{2x}{1-y}) \left(\frac{1-y}{2}\right)^i$ and $g_{ij}(y) = \phat^{2i}_j(y)$. 
The gradients of the auxiliary functions $f_i(x,y)$ and $g_{ij}(y)$ can be calculated as
\begin{equation}\label{eq:DualGrad}
\begin{aligned}
\Grad(f_i(x,y)) &= \left(\dfrac{1-y}{2}\right)^{(i-1)} \begin{pmatrix} L_{i-1}\left(\frac{2x}{1-y}\right)  \\ \frac12 L_{i-2}\left(\frac{2x}{1-y}\right)\end{pmatrix}  &&\text{ for } i\geq2,\\
\Grad(g_{ij}(y)) &= \begin{pmatrix} 0 \\ P^{(2i,0)}_{j-1}(y) \end{pmatrix} &&\text{ for } i\geq 2, j \geq 1, \\
\Grad(f_1(x,y)) &=\dfrac{1-y}{4} \begin{pmatrix} 1  \\\frac{1}{2} \frac{2x}{(1-y)}\end{pmatrix}, 
\end{aligned}
\end{equation}
where we simplified the first gradient as shown in \cite{beuchler}. 
We follow the ansatz as described for the quadrilateral case. 
First split $v^{\triangle,I/II/III}_{ij}$ in the functions
\begin{equation}\label{eq:TrigBase}
\begin{aligned}    
    \tilde{v}^{\triangle,I}_{ij}(x,y) &= \Grad(f_i(x,y)) g_{ij}(y) \quad \text{ for } i \geq 1, j\geq 2,\\
    \tilde{v}^{\triangle,II}_{ij}(x,y) &= f_i(x,y) \Grad(g_{ij}(y)) \quad \text{ for } i,j\geq 2.\\
\end{aligned}
\end{equation}
The functions $\lbrace \tilde{v}_{ij}^{\triangle,I}(x,y)\rbrace_{ij}$ and $\lbrace\tilde{v}_{ij}^{\triangle,II}(x,y)\rbrace_{ij}$ are also a basis of the space $H_0(\Curl)$. Next we derive the biorthogonal vectorial functions for $\tilde{v}^{\triangle,I}_{ij}(x,y)$ and $\tilde{v}^{\triangle,II}_{ij}(x,y),$ and then solve the original problem by linear combination, as in the quadrilateral case. \\
Here the main idea of the construction is that we first find vectorial functions which are orthogonal to either $\Grad(g_{ij}(x,y))$ or $\Grad(f_i(x,y))$ and then biorthogonalise those to the respective other basis functions.\\
It is clear that we have the following structure of the orthogonal vectors:
\begin{equation}\label{Dual:TrigStruc}
    \begin{aligned}
    B_{kl}(x,y) &= \begin{pmatrix} b_{kl}(x,y)\\ 0\end{pmatrix} \quad\textrm{and}\quad 
    C_{kl}(x,y) &= \begin{pmatrix} c_{1,kl}(x,y) \\ c_{2,kl}(x,y) \end{pmatrix}.
    \end{aligned}
\end{equation}
In the following we use the notation $\eta = \frac{2x}{1-y}$ and write all functions in dependence of $(\eta,y)$, e.g. write $a(x,y)$ as $a(\eta,y).$ 
The first problem which needs to be solved then reads:\\
\begin{myprob}
Find polynomials $B_{kl}(\eta,y)$ such that
\begin{align*}
    \langle\tilde{v}_{ij}^{\triangle,II}(\eta,y), B_{kl}(\eta,y)\rangle = 0 \text{ and } \langle \tilde{v}^{\triangle,I}, B_{kl}(\eta,y) \rangle = d^{(1)}_{ijkl} \delta_{ik} \delta_{jl}.
\end{align*}
\end{myprob}
Since the first component of $\tilde{v}^{\triangle,II}(\eta,y)$ is zero, $B_{kl}(\eta,y)$ as in \eqref{Dual:TrigStruc} naturally fulfils the first condition. Furthermore, we can assume a tensorial-like structure, i.e.
\begin{equation*}
    B_{kl}(\eta,y) = \begin{pmatrix} b_{k}^{(1)}(\eta) b_{kl}^{(2)}(y) \\ 0 \end{pmatrix}.
\end{equation*}
Now $b^{(1)}_k(z)$ and $b^{(2)}_{kl}(y)$ only needs to fulfil the relationship
\begin{equation*}
    \int_{-1}^{1} L_{i-1}(\eta)~b^{(1)}_k(\eta) \mathrm{d}\eta \int_{-1}^{1} \left(\frac{1-y}{2}\right)^{i} \frac{(1+y)}{2j} P^{(2i-1,1)}_{j-1}(y)~b_{kl}^{(2)}(y) \dy = d^{(1)}_{ijkl} \delta_{ik} \delta_{jl},
\end{equation*}
where we applied the Duffy transformation. This motivates the choice $b^{(1)}_k(\eta) = L_{k-1}(\eta)$ and $b^{(2)}_{kl}(y) = \left(\frac{1-y}{2}\right)^{k-1} P_{l-1}^{(2k-1,1)}(y).$\\
For the second type of dual functions, we need to solve the following problem:
\begin{myprob}\label{Dual:ProbTrig2}
Find polynomials $C_{kl}(\eta,y)$ such that,
\begin{align}
    \langle \tilde{v}_{ij}^{\triangle,I}(\eta,y), C_{kl}(\eta,y)\rangle = 0 \text{ and } \langle \tilde{v}_{ij}^{\triangle,II}(\eta,y), C_{kl}(\eta,y) \rangle = d^{(2)}_{ijkl} \delta_{ik} \delta_{jl}.
\end{align}
\end{myprob}
We will need the following auxiliary lemma.
\begin{lemma}{}{}\label{lemma:IntJacLeg}
For $1 \leq i \leq k,$ the relation
\begin{equation*}
\int_{-1}^{1} L_i(x) P_k^{(1,1)}(x)\dx = \begin{cases} \frac{4}{2+k} & \text{ if } k\geq i \text{ and } (k-i) \operatorname{mod} 2 = 0 \\0 & \text{else} \end{cases}
\end{equation*}
holds.
\end{lemma}
\begin{proof}
A classical result of Jacobi polynomials states that $P_n^{(\alpha,\alpha)}(x)$ is even if $n$ is even, and it is odd if $n$ is odd. Thus, the relation is trivial if $k$ and $i$ have a different parity.\\
In the following, assume $i$ and $k$ have the same parity. Let $I_{ik} \coloneqq \int_{-1}^{1} L_i(x) P_k^{(1,1)}(x)\dx.$ If $i>k$ it follows that $I_{ik}$ is zero, due to the orthogonality condition of $L_i(x).$
Now assume $k \geq i.$ By partial integration it follows
\begin{align*}
    I_{ik} = \int_{-1}^{1} L_i(x) P_k^{(1,1)}(x)\dx &= \int_{-1}^{1} L_i(x) \frac{2}{2+k}\frac{\mathrm{d}}{\dx}L_{k+1}(x)\dx\\
    &= \frac{2}{2+k}\left[ L_i(x) L_{k+1}(x)\right]\big\vert_{-1}^{1} - \int_{-1}^1 \left(\frac{\mathrm{d}}{\dx} L_i(x)\right) L_{k+1}(x) \dx\\
    &= \frac{4}{2+k},
\end{align*}
where the last integral vanishes due to the orthogonality condition of $L_{k+1}(x)$ and $[L_i(x)L_{k+1}(x)]\big\vert_{-1}^1 = 2$ due to the odd parity.
\end{proof}
\begin{figure}
    \centering
    \begin{subfigure}[t]{0.25\textwidth}
    \includegraphics[scale=0.25]{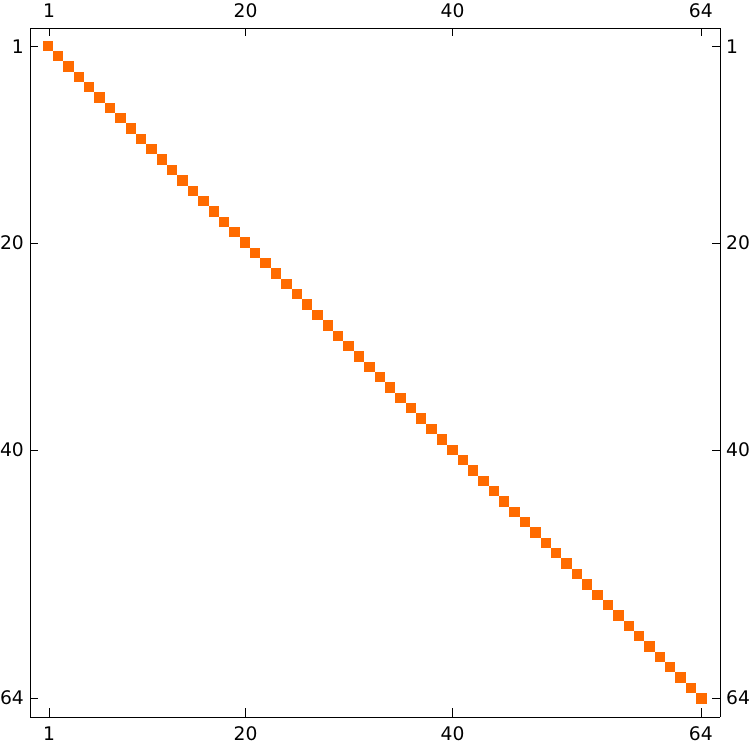}
    \caption{$L^2$-scalar product of $B_{kl},C_{kl}$ and auxiliary functions.}\label{fig:biortha}
    \end{subfigure}
    \begin{subfigure}[t]{0.25\textwidth}
    \includegraphics[scale=0.25]{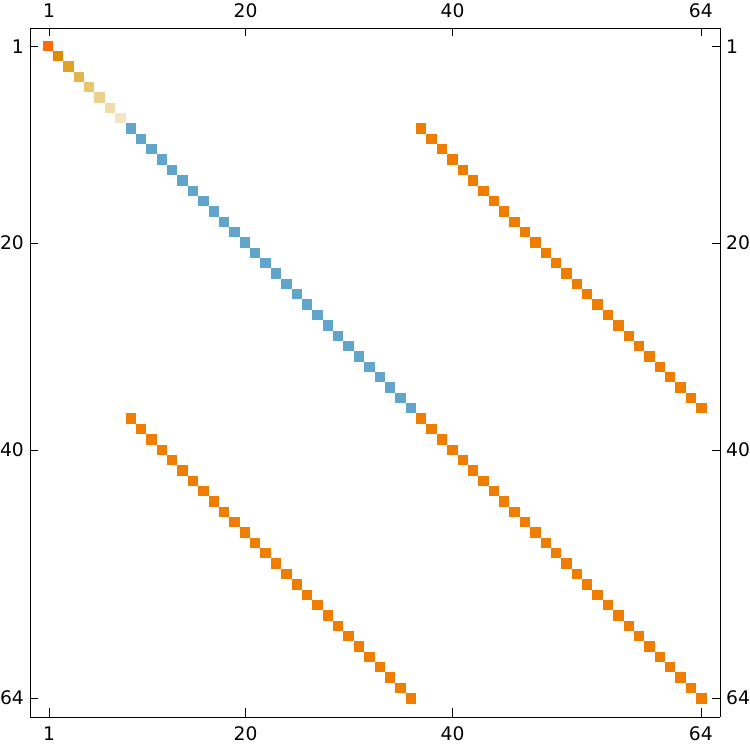}
    \caption{$L^2$-scalar product of $B_{kl},C_{kl}$ and $H(\Curl)$ basis functions.}\label{fig:biorthb}
    \end{subfigure}
    \caption{Biorthogonal sparsity pattern for $p=6$}
    \label{fig:biorth}
\end{figure}
For \cref{Dual:ProbTrig2}, we start with the biorthogonality condition. We again assume a tensorial-like structure, i.e.
        $C_{kl}(\eta,y) =  c_{kl}^{(3)}(y) \begin{pmatrix}c_{k}^{(1)}(\eta) & c_{k}^{(2)}(\eta)\end{pmatrix}^\top$.
The condition is
\begin{align*}
    \langle \tilde{v}^{\triangle,II}_{ij}(\eta,y), C_{kl}(\eta,y) \rangle &= \int_{-1}^1 \lhat_i(\eta) c_{k}^{(2)}(\eta) \deta \int_{-1}^1 \left(\frac{1-y}{2}\right)^{i+1} P^{(2i,0)}_{j-1}(y) c_{kl}^{(3)}(y) \dy\\&= d^{(2)}_{ijkl} \delta_{ik}\delta_{jl}.
\end{align*}
This leads to the choice $c_{k}^{(2)}(\eta) = \kappa P^{(1,1)}_{k-2}(\eta)$ and $c^{(3)}_{kl}(y) = \left(\frac{1-y}{2}\right)^{k-1}P^{(2k,0)}_{l-1}(y),$ where $\kappa$ is some constant. Condition   $ \langle \tilde{v}_{ij}^{\triangle,I},C_{kl}\rangle =0 $
is satisfied, if
     $\langle \Grad f_i(\eta,y) , C_{kl}(\eta) \rangle = 0$.
Since both components of $\Grad f_{i}$ depend on $\left(\frac{1-y}{2}\right)^{i-1} \phat_j^{2i,0}(y),$ the orthogonality relation reduces to
\begin{equation*} \int_{-1}^1 L_{i-1}(\eta)c^{(1)}_{k}(\eta) \deta + \int_{-1}^{1} \frac12 L_{i-2}(\eta) c_{k}^{(2)}(\eta) \deta = 0. \end{equation*}
Due to lemma \ref{lemma:IntJacLeg} this condition is fulfilled if $c_{k}^{(1)}(\eta) = (2+k-1) P_{k-1}^{(1,1)}(\eta)$ \\and ${c_{k}^{(2)}(\eta) = -2(2+k-2) P_{k-2}^{(1,1)}(\eta).}$ Now we have found  biorthogonal functions for $\tilde{v}_{ij}^{\triangle,I}$ and $\tilde{v}_{ij}^{\triangle,II},$ which results in a diagonal matrix which can be seen in \Cref{fig:biortha}. \\
\begin{definition}[Dual functions on a triangle]
For $k \geq 2, l \geq 1$ define 
\begin{equation}\label{Tab:DualTrig}
    \begin{aligned}
B_{kl}(x,y) &\coloneqq \begin{pmatrix}L_{k-1}\dufx\left(\frac{1-y}{2}\right)^{k-1}P_{l-1}^{2k-1,1}(y)\\0\end{pmatrix},\\
C_{kl}(x,y) &\coloneqq \begin{pmatrix}(2+k-1) P_{k-1}^{(1,1)}\dufx\left(\frac{1-y}{2}\right)^{k-1} P_{l-1}^{(2k,0)}(y)\\ -2(2+k-2) P_{k-2}^{(1,1)}\dufx\left(\frac{1-y}{2}\right)^{k-1} P_{l-1}^{(2k,0)}(y)\end{pmatrix}.
    \end{aligned}
\end{equation}
\end{definition}
With this choice we have proven the following corollary.
\begin{corollary}\label{Dual:TrigCor}
Let $\tilde{v}^{\triangle,I}_{ij},\tilde{v}^{\triangle,II}_{ij}$ be defined by \eqref{eq:TrigBase} and the functions $B_{kl}$ and $C_{kl}$ be defined as in \cref{Tab:DualTrig}. Then these functions are biorthogonal.
\end{corollary}
Obviously those are not biorthogonal to the basis $v^{\triangle,I}_{ij}, v^{\triangle,II}_{ij}$ and $v^{\triangle,III}_{1j},$  which can be seen in \Cref{fig:biorthb}. Thus, we derive the biorthogonal functions by linear combination, as in the quadrilateral case. 
\begin{lemma} Let $\alpha_1 = \frac{1}{8} (2i-1)(2j+2i-1)(j+2i-1)$, 
    $\alpha_2 = \frac{1}{16} (2i-1)(2j+2i-1)$ and
    $\alpha_3 = \frac{1}{16}(2j+2)(j+2)$.
Then, for $2 \leq i,k \leq p, 1\leq j,l \leq p,$ and $i+j, k+l \leq p,$ the functions $v^{\triangle,I}_{ij},v^{\triangle,II}_{ij}$ and $v^{\triangle,III}_{1j}$ as in \eqref{Dual:BasisTrig} are biorthogonal to 
\begin{equation*}
\begin{aligned}
    b^{\triangle,I}_{kl}(x,y) &= -\frac12 (\alpha_1 B_{kl}(x,y) + \alpha_2 C_{kl}(x,y)),\\
    b^{\triangle,II}_{kl}(x,y) &= \frac12 (\alpha_1 B_{kl}(x,y) - \alpha_2 C_{kl}(x,y)),\\
b^{\triangle,III}_{1l}(x,y) &= \alpha_3 B_{1l}(x,y),
\end{aligned}
\end{equation*}
where $B_{kl}$ and $C_{kl}$ are given in \cref{Tab:DualTrig}. 
\end{lemma}
\begin{proof}
Since we have shown biorthogonality of $\tilde{v}_{ij}^{\triangle,I},\tilde{v}_{ij}^{\triangle,II}$ to $B_{kl}, C_{kl}$ in \cref{Dual:TrigCor}, we get
\begin{equation*}
    \begin{aligned}
    \langle v^{\triangle,I}_{ij}, B_{kl}\rangle = \langle \tilde{v}^{\triangle,I}_{ij}, B_{kl} \rangle + \langle \tilde{v}^{\triangle,II}_{ij}, B_{kl} \rangle = \langle \tilde{v}^{\triangle,I}_{ij}, B_{kl} \rangle = c \delta_{ik} \delta_{jl}\\
    \langle v^{\triangle,I}_{ij}, C_{kl}\rangle = \langle \tilde{v}^{\triangle,I}_{ij}, C_{kl} \rangle + \langle \tilde{v}^{\triangle,II}_{ij}, C_{kl} \rangle = \langle \tilde{v}^{\triangle,II}_{ij}, C_{kl} \rangle = \tilde{c} \delta_{ik} \delta_{jl}.
    \end{aligned}
\end{equation*}
As in the quadrilateral case, lemma \ref{modBi} with the matrix $A=\begin{bmatrix} 1& 1 \\ 1 & -1\end{bmatrix}$ is applied. For the diagonal matrix $D$ in this lemma,
we have to compute $\langle \tilde{v}^{\triangle,I}_{ij}, B_{ij} \rangle $ and
$\langle \tilde{v}^{\triangle,II}_{ij}, C_{ij} \rangle $
for $(i,j) = (k,l)$. Using \eqref{Orthogonality} with $\alpha=\beta=0$ and $(\alpha,\beta)=(2i-1,0)$  
one obtains
\begin{equation*}
    \begin{aligned}
        \langle \tilde{v}^{\triangle,I}_{ij}, B_{ij} \rangle &= \int_{-1}^{1} (L_{i-1}(\eta))^2 \deta\int_{-1}^{1} \left(\frac{1-y}{2}\right)^{2i-1} \frac{(1+y)}{j} (P_{j-1}^{(2i-1,1)}(y))^2 \dy\\
        &= \frac{8}{(2i-1)(2j+2i-1)(j+2i-1)}.
    \end{aligned}
\end{equation*}
This gives $\alpha_1 = \frac{1}{8} (2i-1)(2j-2i-1)(j+2i-1)$. 
Analogously \eqref{Orthogonality} with $\alpha=\beta=1$ and $(\alpha,\beta)=(2i,0)$ gives
\begin{equation*}
    \begin{aligned}
        \langle \tilde{v}^{\triangle,II}_{ij}, C_{ij} \rangle &= -2i \int_{-1}^{1}\left(\frac{\eta^2-1}{2(i-1)}\right) (P^{(1,1)}_{i-2}(\eta))^2 \deta \int_{-1}^{1} \left(\frac{1-y}{2}\right)^{2i} P^{(2i,0)}_{j-1}(y) \dy\\
        &= \frac{16}{(2i-1)(2j+2i-1)},
    \end{aligned}
\end{equation*}
implies $\alpha_2 = \frac{1}{16}(2i-1)(2j+2i-1).$ The only thing remaining to show is that $B_{kl},C_{kl}$ are orthogonal to $v^{\triangle,III}_{1,j}$ and that $b_{1l}^{\triangle,III}$ is orthogonal to $\tilde{v}^{\triangle,I}_{ij}$ and $\tilde{v}^{\triangle,II}_{ij}.$
Indeed,
\begin{equation*}\langle v^{\triangle,III}_{1j},B_{kl}\rangle = \int_{-1}^{1} 1 \cdot L_{k-1}(\eta) \deta \int_{-1}^{1} \frac12 \left(\frac{1-y}{2}\right)^{k} P_{l-1}^{(2k-1,1)}(y) = 0, \end{equation*}
due to the orthogonality of $L_{k-1}(\eta)$ for all $k\geq 2.$ For the orthogonality of
$\langle v^{\triangle,III}_{1j},C_{kl}\rangle$.
we can apply \cref{lemma:IntJacLeg}.
On the other hand, orthogonality of $b_{1l}^{\triangle,III}$ follows from \cref{Dual:TrigCor}.
\end{proof}
Since all basis functions are properly scaled, the next corollary follows directly.
\begin{corollary}\label{dual:nedcor}
Let $i,k \geq 2$ and $j,l \geq 1.$ Furthermore let $v^{\triangle,I}_{ij}, v^{\triangle,II}_{ij}, v^{\triangle,III}_{1j}$ be the basis of the $H(\Curl)$ interior functions and let $b^{\triangle,I}_{kl},b^{\triangle,II}_{kl},b^{\triangle,III}_{1l}$ be the corresponding normalized dual face functions, then holds for the entries of the element Matrix $G$ that
\begin{equation*}
    G_{ij,kl} = \langle v^{\triangle,T_1}_{ij}(x,y), b_{kl}^{\triangle,T_2}(x,y) \rangle = \delta_{i,k} \delta_{j,l} \delta_{T_1,T_2}.
\end{equation*}
\end{corollary}
\subsection{Tetrahedral case}
Our ansatz is the same as in the triangular case.
Firstly, we split the basis functions in a simpler basis, then find the orthogonal basis to this simpler basis, and lastly build the right dual basis by linear combination.
Recall the basis of $H(\Curl)$ on the tetrahedral $\blacktriangle$ is given by
\begin{equation}\label{Dual:TetBasis}
    \begin{aligned}
    v_{ijk}^I &= \Grad(f_i g_{ij} h_{ijk}) = \Grad (f_i) g_{ij} h_{ijk} + f_i \Grad(g_{ij}) h_{ijk} + f_i g_{ij} \Grad(h_{ijk}),\\
    v_{ijk}^{II} &= \Grad (f_i) g_{ij} h_{ijk} - f_i \Grad(g_{ij}) h_{ijk} + f_i g_{ij} \Grad(h_{ijk}),\\
    v_{ijk}^{III} &= \Grad (f_i) g_{ij} h_{ijk} + f_i \Grad(g_{ij}) h_{ijk} - f_i g_{ij} \Grad(h_{ijk}),\\
    v_{1jk}^{IV} &= v^{\ned_0}_{[1,2]} g_{ij} h_{ijk},
    \end{aligned}
\end{equation}
for $2 \leq i, 1\leq j,k,$ and $i+j+k\leq p,$ where
   $ f_i(x,y,z) = \lhat_i\left(\frac{4x}{1-2y-z}\right) \left(\frac{1-2y-z}{4}\right)^i$,
    $g_{ij}(y,z) = \phat_j^{2i}\left(\frac{2y}{1-z}\right) \left(\frac{1-z}{2}\right)^j,$ and
    $h_{ijk} (z)= \phat_k^{2i+2j}(z).$
Here $v_{[1,2]}^{\ned_0}$ is the lowest order N\'ed\'elec function of first kind, based on the edge from vertex $1$ to $2$. With the substitutions $\eta = \frac{4x}{1-2y-z}$ and $\chi = \frac{2y}{1-z},$ the gradients of the auxiliary functions are
\begin{equation*}
    \begin{aligned}
    \Grad f_i &= \begin{pmatrix} L_{i-1}(\eta) \\ \frac12 L_{i-2}(\eta)\\ \frac14 L_{i-2}(\eta) \end{pmatrix}\left(\frac{1-\chi}{2}\right)^{i-1}\left(\frac{1-z}{2}\right)^{i-1},\\
    \Grad g_{ij} &= \begin{pmatrix}0\\ P_{j-1}^{(2i,0)} (\chi)\\ \frac{\chi}{2} P_{j-1}^{(2i,0)}(\chi) - \frac{j}{2} \phat_{j}^{2i}(\chi) \end{pmatrix} \left(\frac{1-z}{2}\right)^{j-1} \text{ and }\\
    \Grad h_{ijk} &= \begin{pmatrix} 0\\0 \\ P^{(2i+2j,0)}_{k-1}(z)\end{pmatrix},
    \end{aligned}
\end{equation*}
with the simplifications as in \cite{beuchler,beuchlerpillwein}.
We now derive three biorthogonal vectors with respect to $(\Grad f_i)\,g_{ij}\,h_{ijk}, f_i\,(\Grad g_{ij})\,h_{ijk}$ and $f_i\,g_{ij}\,(\Grad h_{ijk}).$\\
Similar to the triangular case, we have the following conditions on the dual functions $B_{ijk},C_{ijk}$ and $D_{ijk}:$
\begin{myprob}
Find $B_{lmn},C_{lmn}, D_{lmn}$ such that
\begin{align*}
      f_i\,(\Grad g_{ij})\,h_{ijk} \perp B_{lmn} \perp f_i\,g_{ij}\,(\Grad h_{ijk}) \text{ and } \langle B_{lmn},(\Grad f_i)\,g_{ij}\,h_{ijk} \rangle = r_{ijklmn} \delta_{il} \delta_{jm} \delta_{kn},\\
      (\Grad f_i)\,g_{ij}\,h_{ijk} \perp C_{lmn} \perp f_i\,g_{ij}\,(\Grad h_{ijk}) \text{ and } \langle C_{lmn}, f_i\,(\Grad g_{ij})\,h_{ijk} \rangle = s_{ijklmn} \delta_{il} \delta_{jm} \delta_{kn},\\
      (\Grad f_i)\,g_{ij}\,h_{ijk} \perp D_{lmn} \perp f_i\,(\Grad g_{ij})\,h_{ijk} \text{ and } \langle D_{lmn},f_i\,g_{ij}\,(\Grad h_{ijk}) \rangle = t_{ijklmn} \delta_{il} \delta_{jm} \delta_{kn}.
\end{align*}
\end{myprob}

Since the basis vectors build a lower triangular system, we know that we need to find an upper triangular system. Thus,
\begin{equation*}
\begin{aligned}
    B_{lmn} = \begin{pmatrix} b_{lmn}(\eta,\chi,z) \\ 0 \\ 0 \end{pmatrix},\quad
    C_{lmn} = \begin{pmatrix} c^{(1)}_{lmn}(\eta,\chi,z) \\ c^{(2)}_{lmn}(\eta,\chi,z) \\ 0 \end{pmatrix},\quad\textrm{and}\quad
    D_{lmn} = \begin{pmatrix} d^{(1)}_{lmn}(\eta,\chi,z) \\ d^{(2)}_{lmn}(\eta,\chi,z) \\ d^{(3)}_{lmn}(\eta,\chi,z) \end{pmatrix}.
\end{aligned}
\end{equation*}
The construction of $B_{lmn}$ is trivial, and the construction of $C_{lmn}$ follows from the $2D$ case. Thus,
\begin{equation*}
\begin{aligned}
    B_{lmn} &= \begin{pmatrix} 1 \\ 0 \\ 0 \end{pmatrix} L_{l-1}(\eta) \left(\frac{1-\chi}{2}\right)^{l-1} P_{m-1}^{(2l-1,1)}(\chi) \left(\frac{1-z}{2}\right)^{l+m-2}P_{n-1}^{(2l+2m-1,1)}(z) \text{ and }\\
    C_{lmn} &= \begin{pmatrix} (2+l-1)P^{(1,1)}_{l-1}(\eta) \\ -2(2+l-2)P^{(1,1)}_{l-2}(\eta) \\ 0 \end{pmatrix} \left(\frac{1-\chi}{2}\right)^{l-1}P_{m-1}^{(2l,0)}(\chi) \left(\frac{1-z}{2}\right)^{l+m-2} P_{n-1}^{(2l+2m-1,1)}(z),
\end{aligned}
\end{equation*}
where the exponents of $\left(\frac{1-\chi}{2}\right)$ and $\left(\frac{1-z}{2}\right)$ are determined with respect to the functional determinant of the Duffy trick, i.e. $\left(\frac{1-\chi}{2}\right)\left(\frac{1-z}{2}\right)^2.$\\
To derive $D_{lmn},$ we go step by step. It follows immediately that 
\begin{equation}d^{(3)}_{lmn} = P_{l-2}^{(1,1)}(\eta) \left(\frac{1-\chi}{2}\right)^{l-2} P_{m-1}^{(2l-1,1)}(\chi) \left(\frac{1-z}{2}\right)^{l+m-2} P_{n-1}^{(2l+2m,0)}(z),\end{equation}
due to $\langle D_{lmn},f_i\,g_{ij}\,(\Grad h_{ijk}) \rangle = s_{ijklmn} \delta_{il} \delta_{jm} \delta_{kn}.$
Next we derive $d^{(2)}_{lmn}$ by demanding 
\begin{equation}\label{eq:DualCondI}
    0 = \langle f_{i}(\Grad g_{ij}) h_{ijk}, D_{lmn} \rangle.
\end{equation}
An rather obvious choice for $d^{(2)}_{lmn}$ is $d^{(2)}_{lmn} = P^{(1,1)}_{l-2} (\eta) \, \tilde{d}^{(2)}(\chi) \left(\frac{1-z}{2}\right)^{l+m-2} P_{n-1}^{(2l+2m,0)}(z).$
This choice reduces the condition \eqref{eq:DualCondI} to
\begin{equation}\begin{aligned}\label{eq:IntJacElem}
    0 =& \int_{-1}^{1} \left(\frac{1-\chi}{2}\right)^{i+1} P_{j-1}^{(2i,0)}(\chi) \,\tilde{d}^{(2)}(\chi)\\ &+ \left(\frac{1-\chi}{2}\right)^{2i-1}\left(\frac{\chi}{2} P_{j-1}^{(2i,0)}(\chi) -\frac{j}{2} \phat_j^{2i}(\chi)\right)P_{m-1}^{(2i-1,1)}(\chi) \,\mathrm{d}\chi,
\end{aligned}\end{equation}
since the condition \eqref{eq:DualCondI} is trivial for $i \neq l$ and $k\neq n$ with this choice of $d^{(2)}_{ijk}.$ \\
On the right part of the integral, the combination $\phat^{2i}_j(\chi)P_{m-1}^{(2i-1,1)}(\chi)$ already fulfils the orthogonality relation. On the other hand, for the product of the two different Jacobi polynomials, namely $P^{(2i,0)}_{j-1}(\chi)$ and $P_{m-1}^{(2i-1,1)}(\chi),$ we can't apply standard orthogonality results. We eliminate this mixed part by linear combination. Therefore, we choose 
\[\tilde{d}^{(2)} = -\frac{\chi}{2}\left(\frac{1-\chi}{2}\right)^{l-2} P_{m-1}^{(2l-1,1)}(\chi) + \hat{d}^{(2)},\]
such that the mixed products cancel each other out. Those linear combinations result in the further reduced condition
\begin{equation}\label{eq:DualCondII}
        0 = \int_{-1}^{1} \left(\frac{1-\chi}{2}\right)^{i+1} P_{j-1}^{(2i,0)}(\chi) \,\hat{d}^{(2)}(\chi) - \frac{j}{2} \left(\frac{1-\chi}{2}\right)^{2i-1}\phat_j^{2i}(\chi)P_{m-1}^{(2i-1,1)}(\chi) \,\mathrm{d}\chi.
\end{equation}
The last part of the integral in \eqref{eq:DualCondII} only appears if $m=j.$ We can achieve the same for the first part of the integral, if we choose $\hat{d}^{(2)} = c \left(\frac{1-\chi}{2}\right)^{l-1}P_{m-1}^{(2l,0)}(\chi)$.
Since both instances in \eqref{eq:DualCondII} are integrals over Jacobi polynomials with matching indices, order and weights, we can determine the constant $c$ directly. It holds that
\begin{align*}
    \int_{-1}^{1} \left(\frac{1-\chi}{2}\right)^{2i} \left(P_{j-1}^{(2i,0)}(\chi)\right)^2 \,\mathrm{d}\chi &= \frac{1}{2j+2i-1}\\
    \int_{-1}^{1} \left(\frac{1-\chi}{2}\right)^{2i-1}\left(\frac{1+\chi}{2}\right) \left(P_{j-1}^{(2i-1,1)}(\chi)\right)^2 \,\mathrm{d}\chi &= \frac{j}{(2j+2i-1)(2i+j-1)}. 
\end{align*}
Collecting everything 
\begin{equation*}
\begin{aligned}
    d^{(2)}_{lmn} = P^{(1,1)}_{l-2}(\eta) &\left(\frac{1-\chi}{2}\right)^{l-2} Q_{m,2}(\chi)
     \left(\frac{1-z}{2}\right)^{l+m-2}P_{n-1}^{(2l+2m,0)}(z)
\end{aligned}
\end{equation*}
with the polynomial
\begin{equation}
    \label{def:Qm2}
    Q_{m,2}(\chi)=-\frac{\chi}{2}P_{m-1}^{(2l-1,1)}(\chi) + \frac{m}{2l+m-1} \frac{1-\chi}{2} P_{m-1}^{(2l,0)}(\chi)
\end{equation}
of degree $m$.
Now we need to determine $d^{(1)}_{lmn},$ by 
    $0 = \langle (\Grad f_i) g_{ij} h_{ijk}, D_{lmn} \rangle$.
Inserting $d^{(2)}_{lmn}$ and $d^{(3)}_{lmn}$ yields the following condition, after some simplification
\begin{align*}
    0 =& \int_{(-1,1)^3}L_{i-1}(\eta)\left(\frac{1-\chi}{2}\right)^{i-1} \phat_{j}^{2i}(\chi) \left(\frac{1-z}{2}\right)^{i+j+1} \phat^{2i+2j}_k(z) d^{(1)}_{lmn}(\eta,\chi,z) \,\mathrm{d}\eta \,\mathrm{d}\chi \,\mathrm{d}z\\ 
    &+ \int_{(-1,1)^3} L_{i-2}(\eta) P_{l-2}^{(1,1)}(\eta) \left(\frac{1-\chi}{2}\right)^{i+l-1} \phat^{2i}_{j}(\chi) \left[P_{m-1}^{(2l-1,1)}(\chi) + \frac{m}{2l+m-1}P_{m-1}^{(2l,0)}(\chi)\right]\\&\left(\frac{1-z}{2}\right)^{i+j+l+m-1} \phat^{2i+2j}_k(z) P^{(2l+2m-1,1)}_{n-1}(z) \, \mathrm{d}z \, \mathrm{d} \chi \, \mathrm{d}\eta.
\end{align*}
If we choose $d^{(1)}_{lmn}(\eta,\chi,z) = \frac{2+l-1}{2(2+l-2)} P_{l-1}^{(1,1)}(\eta) \hat{d}^{(1)}(\chi) \left(\frac{1-z}{2}\right)^{l+m-2} P_{n-1}^{(2l+2m,0)}(z),$ we can factor all terms depending on $\eta$ and $z$ out. Thus, we only need to determine $\hat{d}^{(1)}(\chi),$ by the condition
\begin{align*}
    0 =& \int_{-1}^{1} \left(\frac{1-\chi}{2}\right)^{i-1} \phat_{j}^{2i}(\chi)\,\hat{d}^{(1)}(\chi)\\ &+ \left(\frac{1-\chi}{2}\right)^{i+l-1} \phat^{2i}_{j}(\chi) \left[P_{m-1}^{(2l-1,1)}(\chi) + \frac{m}{2l+m-1}P_{m-1}^{(2l,0)}(\chi)\right]\, \mathrm{d} \chi.
\end{align*} 
It is obvious, that we can choose 
\begin{equation*}
\begin{aligned}
d^{(1)}_{lmn} =& \frac{-(2+l-1)}{2(2+l-2)} P_{l-1}^{(1,1)}(\eta) \left(\frac{1-\chi}{2}\right)^{l-1}
Q_{m-1,1}(\chi) \left(\frac{1-z}{2}\right)^{l+m-2} P_{n-1}^{(2l+2m,0)}(z)
\end{aligned}
\end{equation*}
with the $m-1$ degree polynomial
\begin{equation}\label{def:Qm1}
Q_{m-1,1}(\chi)=P_{m-1}^{(2l-1,1)}(\chi) + \frac{m}{2l+m-1}P_{m-1}^{(2l,0)}(\chi).
\end{equation}
We summarize in the following definition.
\begin{definition}[Dual basis on the tetrahedron]
Let $Q_{m-1,1}$ and $Q_{m,2}$ be defined by \eqref{def:Qm1} and \eqref{def:Qm2}, respectively.
Let $l\geq2, m,n \geq 1$ then define
\begin{align}\label{Tab:DualTet}
    \tilde{b}^{\blacktriangle,I}_{lmn} \coloneqq& \begin{pmatrix} 1 \\ 0 \\ 0 \end{pmatrix} L_{l-1}(\eta) \left(\frac{1-\chi}{2}\right)^{l-1} P_{m-1}^{(2l-1,1)}(\chi) \left(\frac{1-z}{2}\right)^{l+m-2}P_{n-1}^{(2l+2m-1,1)}(z)\\
     \tilde{b}^{\blacktriangle,II}_{lmn} \coloneqq& \begin{pmatrix} (l+1)P^{(1,1)}_{l-1}(\eta) \\ -2l P^{(1,1)}_{l-2}(\eta) \\ 0 \end{pmatrix} \left(\frac{1-\chi}{2}\right)^{l-1}P_{m-1}^{(2l,0)}(\chi) \left(\frac{1-z}{2}\right)^{l+m-2} P_{n-1}^{(2l+2m-1,1)}(z), \nonumber\\
    \tilde{b}^{\blacktriangle,III}_{lmn} \coloneqq& \begin{pmatrix} 
    \frac{-(l+1)}{2\,l} P_{l-1}^{(1,1)}(\eta) \left(\frac{1-\chi}{2}\right)^{l-1} Q_{m-1,1}(\chi) \\ P_{l-2}^{(1,1)}(\eta)\left(\frac{1-\chi}{2}\right)^{l-2} Q_{m,2}(\chi) \\ 
    P_{l-2}^{(1,1)}(\eta) \left(\frac{1-\chi}{2}\right)^{l-2} P_{m-1}^{(2l-1,1)}(\chi) 
    \end{pmatrix} \left(\frac{1-z}{2}\right)^{l+m-2} P_{n-1}^{(2l+2m,0)}(z) 
    \nonumber,
\end{align}
where $\eta = \frac{4x}{1-2y-z}$, $\chi = \frac{1-2y-z}{4}$ .
\end{definition}
Thus, the following lemma has been shown.
\begin{lemma}\label{Dual:tetcor}
Let $f_i, g_{ij}$ and $h_{ijk}$ be defined as in \eqref{Dual:TetBasis}. Then, the functions
\begin{align*}
    \tilde{v}^{\blacktriangle,I}_{ijk} = (\Grad f_i)g_{ij} h_{ijk},\quad
    \tilde{v}^{\blacktriangle,II}_{ijk}= f_i (\Grad g_{ij}) h_{ijk},\quad\textrm{and}\quad
    \tilde{v}^{\blacktriangle,III}_{ijk}= f_i g_{ij} (\Grad h_{ijk})
\end{align*}
 are biorthogonal to $\tilde{b}_{lmn}^{\blacktriangle,I},\tilde{b}_{lmn}^{\blacktriangle,II}$ and $\tilde{b}_{lmn}^{\blacktriangle,III}$ defined by \cref{Tab:DualTet}, i.e.
\[
\langle \tilde{v}^{\blacktriangle,\omega_1}_{ijk}, \tilde{b}^{\blacktriangle,\omega_2}_{lmn} \rangle = c^{\omega_1}_{ijk} \, \delta_{il} \delta_{jm} \delta_{kn} \delta_{\omega_1,\omega_2}, \quad \omega_1,\omega_2 \in \lbrace I, II, III \rbrace.
\]
\end{lemma}
As before, we transfer this to the original interior basis functions. 
\begin{lemma}
Let $\tilde{b}_{lmn}^{\blacktriangle,I},\tilde{b}_{lmn}^{\blacktriangle,I}$ and $\tilde{b}_{lmn}^{\blacktriangle,I}$ be defined by \cref{Tab:DualTet}.\\
For $i \geq 2, j \geq 1, k \geq 1$ the functions $v^{\blacktriangle,I}_{ijk},v^{\blacktriangle,II}_{ijk},v^{\blacktriangle,III}_{ijk}$ and $v^{\blacktriangle,IV}_{1jk}$ are biorthogonal to
\begin{align*}
    b^{\blacktriangle,I}_{lmn} &=  \frac12 \alpha^{(2)}_{lmn} \tilde{b}^{\blacktriangle,II}_{lmn} +\frac12 \alpha^{(3)}_{lmn} \tilde{b}^{\blacktriangle,III}_{lmn},\\
    b^{\blacktriangle,II}_{lmn} &= \frac12 \alpha^{(1)}_{lmn} \tilde{b}^{\blacktriangle,I}_{lmn} - \frac12 \alpha_{lmn}^{(2)} \tilde{b}^{\blacktriangle,II}_{lmn}, \\
    b^{\blacktriangle,III}_{lmn} &= \frac12 \alpha^{(1)}_{lmn} \tilde{b}^{\blacktriangle,I}_{lmn} - \frac12 \alpha_{lmn}^{(3)} \tilde{b}^{\blacktriangle,III}_{lmn},\\
    b^{\blacktriangle,IV}_{1mn} &= \alpha^{(4)}_{lmn} \tilde{b}^{\blacktriangle,I}_{1mn},
\end{align*}
where 
\begin{align*}
    \alpha^{(1)}_{lmn} &= \frac{1}{2^7} (2l-1)(2m+2l-1)(m+2l-1)(2n+2l+2m-1)(n+2l+2m-1)\\
    \alpha^{(2)}_{lmn} &= \frac{1}{2^6} (2l-1)(2m+2l-1)(2n+2l+2m-1)(n+2l+2m-1)\\
    \alpha^{(3)}_{lmn} &= \frac{-1}{2^5} l(2l-1)(2m+2l-1)(m+2l-1) (2n+2l+2m-1)\\
    \alpha^{(4)}_{lmn} &= \frac{1}{2^5} (2m+2)(m+2)(n+2m+2)(2n+2m+2) 
\end{align*}
\end{lemma}\begin{proof}
Again, we apply Lemma \ref{modBi}. Now, 
$A=\begin{bmatrix} 1 & 1 &1 \\ 1 & -1 & 1 \\ 1 & 1 &-1   \end{bmatrix}$.
from which its inverse is easily be computed as
$A^{-1}=\frac12 \begin{bmatrix} 0 & 1 &1 \\ 1 & -1 & 0 \\ 1 & 0 &-1   \end{bmatrix}$.
It remains to compute the diagonal entries.

The coefficients $c^{\omega}_{ijk}$ can be computed analogously to the triangular case be using the exact values of the integrals over Jacobi polynomials.
 Finally, one obtains
\begin{equation*}\begin{aligned}
    \langle \tilde{v}_{ijk}^I, \tilde{b}_{ijk}^I \rangle
    &= \frac{2^7}{(2i-1)} \frac{1}{(2j+2i-1)(j+2i-1)(2k+2i+2j-1)(k+2i+2j-1)},\\
    \langle \tilde{v}_{ijk}^{II}, \tilde{b}_{ijk}^{II} \rangle  &= \frac{2^6}{(2i-1)} \frac{1}{(2j+2i-1)(2k+2i+2j-1)(k+2i+2j-1)},\\
    \langle \tilde{v}_{ijk}^{III}, \tilde{b}_{ijk}^{III} \rangle
    & = \frac{2^6}{(2i-1)} \frac{1}{i \, (2i+2j-1)(j+2i-1)(2k+2i+2j-1) },
    \end{aligned}
\end{equation*}
by using the orthogonality relations of the Jacobi polynomials \eqref{Orthogonality}.
It remains to show, that to
\begin{equation*}
    v^{IV}_{1jk} = \begin{pmatrix} 1\\ \frac{\eta}{2} \\ \frac{\eta}{4} \end{pmatrix} \left(\frac{1-\chi}{2}\right) \phat^{1}_j(\chi) \left(\frac{1-z}{2}\right)^{j+1} \phat^{2j+3}_k(z)
\end{equation*}
the dual shape functions are naturally orthogonal. For $B_{ijk}$ orthogonality follows since the first component of $v^{IV}_{1jk}$ is independent of $\eta = \frac{4x}{1-2y-z}.$\\
For $C_{ijk}$ we apply the relations
\begin{align*}
    \int_{-1}^{1} (i+1) P^{(1,1)}_{i-1}(x)\dx& = 2 \int_{-1}^{1} \frac{\mathrm{d}}{\mathrm{d} x} L_{i}(x) \dx = 2\left[L_{i}(x)\right]_{-1}^{1} = 2(1 - (-1)^i) \quad\textrm{and}\\
    \int_{-1}^{1} \frac{x}{2} (2i) P^{(1,1)}_{i-2}(x)\dx
    &= 2\int_{-1}^{1} x \frac{\mathrm{d}}{\dx} L_{i-1}(x) \dx \\
    &= 2\left[x L_{i}(x)\right]_{-1}^{1} - \underbrace{\int_{-1}^{1} L_{-1}(x)}_{= 0} \dx 
    = 2(1 - (-1)^{i})
\end{align*}
to see that the scalar product $\langle v^{IV}_{1jk},C_{lmn}\rangle = 0,$ for all  $i,j,k,l,m,n.$ For $D_{ijk}$ the same relation is applied, thus $\langle v^{IV}_{1jk}, D_{lmn} \rangle = 0,$ for all $ i,j,k,l,m,n.$
On the other hand the dual function to $v^{IV}_{1jk}$ is easily found to be 
    $B_{1jk} = \begin{pmatrix}1 \\ 0 \\ 0 \end{pmatrix} P_{j-1}^{(2,1)}(y) \left(\frac{1-z}{2} \right)^{j-1} P_{k-1}^{(2j+2,1)}(z)$.
It is obviously orthogonal to $\Grad g_{ij}$ and $\Grad h_{ijk},$ furthermore it is orthogonal to $\Grad f_i$ since it is independent of $\eta.$
\end{proof}
We conclude with some remarks. 
\begin{remark}
It is usually possible to modify the index of the integrated Jacobi polynomials to modify the sparsity pattern and condition number of the element matrices. But in the context of polynomial dual functions the index $(2i)$ and $(2i+2j)$ are minimal, otherwise the dual functions will become rational with singularities for low polynomial degrees for the interior $H(\Curl)$ shape functions.
\end{remark}

\begin{remark}
The coefficients $\alpha_{lmn}^{(1)},\alpha_{lmn}^{(2)}$ and $\alpha_{lmn}^{(3)}$ can be significantly reduced by dividing each with \[(2l\spl2m\sm1)(2l\sm1)(2n\sm2l\sm2m\sm1).\] In this case one needs to compute the element matrix corresponding to biorthogonal system by numerical quadrature or similar methods. 
\end{remark}
\section{\label{sec:Con}Conclusion}
We summarize the main contribution of this paper in a more abstract notation.
Let $\V_0$ be the space $H^1_0$ or $H_0(\Curl)$ on an element $\Omega$ and
$\V_{hp,0}=\V_{hp} \cap \V_0$.
For given element based families of $hp$-FEM basis functions $\phi_i\in \V_{hp,0},$  
the authors developed biorthogonal test functions $\psi_j\in W_{hp}$. This allows us to represent the $L^2$-like projection based interpolation operator
$\mathcal{P}: \V_0(\Omega) \mapsto \V_{hp,0}$ given by
\[
\int_{\Omega}  (\mathcal{P}u)(\vec{x}) \; v_{hp}(\vec{x}) (\vec{x}) \;\mathrm{d}\vec{x} =
\int_{\Omega} u(\vec{x}) v_{hp}(\vec{x})\;\mathrm{d}\vec{x} \quad \forall v_{hp}\in W_{hp}.
\]
by 
\[
(Pu)(\vec{x}) =\sum_{i} \phi_i(\vec{x})  g_i \quad \textrm{with}\quad g_i= 
\int_{\Omega} u(\vec{x}) \psi_i(\vec{x}) \;\mathrm{d}\vec{x}.
\]
The primal functions $\phi_i$ are basis functions which are optimal to sparsity and condition number
on the reference element. The dual functions are expressed in closed expressions in terms
of Jacobi polynomials. This result can be used for further results in error estimation and
preconditioning where the above-mentioned operator $\mathcal{P}$ is used.\\
An extension to other weights or indices of Jacobi polynomials is limited. Since we include the weights of Jacobi polynomials in the basis functions, a decrease in the related indices results in the loss of most orthogonality relations for the low order cases.
On the other hand, an increase of the indices results in a modification of the dual functions and in a higher polynomial test space in \eqref{proInt}. From the point of view of the authors, closed formulas for dual functions of the basis functions of Fuentes et al. \cite{Fuentes} are very ambitious to develop.

\section{Acknowledgment}

\bibliographystyle{siamplain}
\bibliography{references}

\end{document}